%% file: AKM_analyticNP_final.tex
\documentclass[11pt,reqno]{article}
\usepackage{latexsym, amsmath, amssymb, amsthm, a4, epsfig}
\usepackage{graphicx}

\theoremstyle{theorem}
\newtheorem{theorem}{Theorem}[section]

\newtheorem{lemma}[theorem]{Lemma}

\theoremstyle{remark}
\newtheorem{example}{Example}

\newcommand{\p}{\partial}

\newcommand{\eqnref}[1]{(\ref {#1})}

\newcommand{\Cbb}{\mathbb{C}}

\newcommand{\Rbb}{\mathbb{R}}

\newcommand{\Zbb}{\mathbb{Z}}

\newcommand{\la}{\langle}
\newcommand{\ra}{\rangle}

\newcommand{\Hcal}{\mathcal{H}}

\newcommand{\Kcal}{\mathcal{K}}

\newcommand{\Scal}{\mathcal{S}}
\newcommand{\Tcal}{\mathcal{T}}

\newcommand{\Ga}{\alpha}
\newcommand{\Gb}{\beta}

\newcommand{\Ge}{\epsilon}

\newcommand{\Gvf}{\varphi}

\newcommand{\Gk}{\kappa}

\newcommand{\Gl}{\lambda}

\newcommand{\Gv}{\nu}

\newcommand{\Gt}{\theta}

\newcommand{\Gs}{\sigma}

\newcommand{\Gj}{\tau}

\newcommand{\Gy}{\psi}

\newcommand{\GO}{\Omega}

\newcommand{\beq}{\begin{equation}}
\newcommand{\eeq}{\end{equation}}

\def\ol{\overline}

\numberwithin{equation}{section}
\numberwithin{figure}{section}

\begin{document}
\title{Exponential decay estimates of the eigenvalues for the Neumann-Poincar\'{e} operator on analytic boundaries in two dimensions\thanks{This work is supported by A3 Foresight Program among China (NSF), Japan (JSPS), and Korea (NRF 2014K2A2A6000567)}}

\author{Kazunori Ando\thanks{Department of Electrical and Electronic Engineering and Computer Science, Ehime University, Ehime 790-8577, Japan. Email: {\tt ando@cs.ehime-u.ac.jp}.} \and Hyeonbae Kang\thanks{Department of Mathematics, Inha University, Incheon 402-751, S. Korea. Email: {\tt hbkang@inha.ac.kr}.} \and Yoshihisa Miyanishi\thanks{Center for Mathematical Modeling and Data Science, Osaka University, Osaka 560-8631, Japan. Email: {\tt miyanishi@sigmath.es.osaka-u.ac.jp}.}}

\maketitle

\begin{abstract}
We show that the eigenvalues of the Neumann-Poincar\'{e} operator on analytic boundaries of simply connected bounded planar domains tend to zero exponentially fast, and the exponential convergence rate is determined by the maximal Grauert radius of the boundary. We present a few examples of boundaries to show that the estimate is optimal.
\end{abstract}



\section{Introduction}\label{sec:intro}

The Neumann-Poincar\'{e} (NP) operator is an integral operator defined on the boundary of a bounded domain. It arises naturally when solving  the Dirichlet and Neumann boundary value problems for the Laplacian in terms of layer potentials. As the name suggests, its study goes back to Neumann \cite{Neumann-87} and Poincar\'{e} \cite{Poincare-AM-87}. It was a central object in the Fredholm theory of integral equations and the theory of singular integral operators.

In this paper we consider spectral properties of the NP operator. There were some work on spectral properties of the NP operator around 1950s (see, for example, \cite{MR0104934} and references therein). Lately we see rapidly growing interest in spectral properties of the NP operator, which is due its connection to plasmon resonance and cloaking by anomalous localized resonance; see e.g. \cite{PhysRevB.72.155412, MR2263683, MR2186014} and references therein. In fact, in the quasi-static limit, the plasmon resonance takes place at the eigenvalues of the NP operator \cite{AMRZ, AKL-SIAP-16}, and the anomalous localized resonance takes place at the accumulation point of eigenvalues \cite{ACKLM, AK-JMAA-16}.

Recently, there has been considerable progress on the spectral theory of the NP operator. In \cite{KPS} Poincar\'{e}'s variational problem has been revisited with modern language of mathematics. Among findings of the paper is that the NP operator can be symmetrized by introducing a new (but equivalent) inner product to $H^{-1/2}$ space, the Sobolev $-1/2$ space. This is a quite important discovery for the spectral theory of the NP operator. The NP operator, as a self-adjoint operator, has only two kinds of spectra: continuous spectrum and discrete spectrum (see, for example, \cite{Yosida}). If a given domain has a smooth boundary, then the NP operator is compact and has only eigenvalues accumulating to $0$. If the boundary has a corner, then the NP operator is a singular integral operator, and may have continuous spectrum. We refer to \cite{HKL, KLY, PP1, PP2} for recent developments on the NP spectral theory on planar domains with corners. Here and afterwards, the \emph{NP spectrum} is an abbreviation of the spectrum of the NP operator.

As mentioned above, if the domain has a smooth, $C^{1,\Ga}$ ($\Ga>0$) to be precise, boundary, then the NP operator is compact and has eigenvalues converging to $0$. In the recent paper \cite{Miyanishi:2015aa}, a quantitative estimate of the decay rate of NP eigenvalues has been obtained:
Let $\{ \Gl_j \}$ be the NP eigenvalues arranged in such a way that $| \Gl_1 | = | \Gl_2| \ge | \Gl_3 | = | \Gl_4 | \ge \cdots$.
It is proved that if the boundary of the domain is $C^k$ ($k \ge 2 $), then
\beq
|\Gl_n| = o(n^{\Ga}) \quad \text{ as } n \to \infty,
\eeq
for any $\Ga > - k + 3/2$. So, if, in particular, the boundary is $C^{\infty}$ smooth, then NP eigenvalues decay faster than any algebraic order.
On the other hand, the NP eigenvalues on the ellipse of the long axis $a$ and the short axis $b$ are known to be
\beq\label{eigell}
\pm \frac{1}{2} \left( \frac{a - b}{a + b} \right)^n , \quad n=1,2,\ldots.
\eeq
So, one may suspect that NP eigenvalues on analytic boundaries tend to $0$ exponentially fast. We prove it in this paper.

We show that if the boundary is analytic, then NP eigenvalues converge to $0$ exponentially fast and the exponential convergence rate is determined by the modified maximal Grauert radius. See Theorem \ref{thm:3} for the precise statement of the result and subsection \ref{sec:grauert} for the definition of the modified maximal Grauert radius. We do not know if the convergence rate is optimal in general. However, we show that it is optimal on domains like disks, ellipses, and
lima\c{c}ons of Pascal. It is worth emphasizing that the main theorem is proved using the Weyl-Courant min-max principle and a Paley-Wiener type lemma (Lemma \ref{lem:2}).

This paper is organized as follows. In section \ref{sec:prelim} we review symmetrization of the NP operator, define the modified maximal Grauert radius (and tube), and show that the integral kernel of the NP operator admits analytic continuation to the modified maximal Grauert tube. Section \ref{sec:Complex_Fourier} is to present and prove the main result of this paper. Section \ref{sec:Examples} is to present some examples to show that the modified maximal Grauert radius yields the best possible bound for the convergence.

\section{Preliminaries}\label{sec:prelim}
\subsection{The NP operator and symmetrization}

Once for all, we assume that $\GO$ is a bounded planar domain whose boundary, $\p\GO$, is analytic.
The single layer potential of  a function $\Gvf$ on $\p \GO$ is defined by
$$
\Scal_{\p \GO} [\Gvf] (x) = \frac{1}{2\pi} \int_{\p \GO} \ln |x-y| \Gvf(y) d\Gs(y), \quad x \in \Rbb^2,
$$
where $d\Gs$ is the length element of $\p \GO$.
The NP operator on $\p \GO$ is defined by
$$
\Kcal_{\p \GO}^* [\Gvf] (x) = \frac{1}{2 \pi} \int_{\p \GO} \frac{\la x - y, \Gv_x \ra}{| x - y |^2}  \Gvf(y) d\Gs(y), \quad x \in \p \GO,
$$
where $\nu_x$ is the outward unit normal vector at $x \in \p\GO$. The relation between the NP operator and the single layer potential is given by the jump relation for which we refer to, for example, \cite{book2}. It is well-known that the Plemelj's symmetrization principle holds:
\beq\label{plem}
\Scal_{\p \GO} \Kcal_{\p \GO}^* = \Kcal_{\p \GO} \Scal_{\p \GO},
\eeq
where $\Kcal_{\p \GO}$ is the $L^2(\p \GO)$-adjoint of $\Kcal_{\p \GO}^*$.

We denote by $H^s = H^s(\p \GO)$, $s \in \Rbb$, the usual Sobolev space on $\p \GO$ and its norm is denoted by $\| \cdot \|_s$.
Define
\beq
  \la \Gvf, \Gy \ra_{\Hcal^*} := - \la \Gvf, \Scal_{\p \GO} [\Gy] \ra_{L^2} \label{eq:new_inner_product1}
\eeq
for $\varphi, \psi \in H_0^{- 1 / 2} := \{ \Gvf \in H^{- 1 / 2}; \la \Gvf, 1 \ra_{L^2} = 0 \}$.
We emphasize that the right hand side of \eqref{eq:new_inner_product1} is well-defined since $\Scal_{\p\GO}$ (as an operator on $\p\GO$) maps $H^{- 1 / 2}$ to $H^{1 / 2}$.
It is known that $\la \cdot, \cdot \ra_{\Hcal^*}$ is an inner product on $H_0^{- 1 / 2}$ and induces the norm equivalent to $\| \cdot \|_{-1/2}$, namely, there are constants $C_1$ and $C_2$ such that
\beq\label{twonorm}
C_1 \| \Gvf \|_{- 1 / 2} \le \| \Gvf \|_{\Hcal^*} \le C_2 \| \Gvf \|_{- 1 / 2}
\eeq
for all $\Gvf \in H_0^{- 1 / 2}$ (see \cite{KKLSY}).

We put $\Hcal_0^* := H_0^{- 1 / 2}$ equipped with the inner product $\la \cdot, \cdot \ra_{\Hcal^*}$.
Then $\Kcal_{\p \GO}^*$ is a self-adjoint operator on $\Hcal_0^*$. In fact, we have from \eqnref{plem}
$$
\la \Gvf, \Kcal_{\p \GO}^*[\psi] \ra_{\Hcal^*} = \la \Gvf, \Scal_{\p \GO} \Kcal_{\p \GO}^*[\psi] \ra_{L^2} = \la \Gvf, \Kcal_{\p \GO} \Scal_{\p \GO}[\psi] \ra_{L^2} = \la \Kcal_{\p \GO}^*[\Gvf], \psi \ra_{\Hcal^*}.
$$
So, as a self-adjoint compact operator on a Hilbert space, $\Kcal_{\p \GO}^*$ has eigenvalues converging to $0$. It is known that all eigenvalues lie in $(- 1 / 2, 1 / 2)$ (see \cite{MR0222317}). It is worth mentioning that $1/2$ is an eigenvalue of $\Kcal_{\p \GO}^*$ of multiplicity $1$ if we consider $\Kcal_{\p \GO}^*$ as an operator on $H^{- 1 / 2}$, not on $H_0^{- 1 / 2}$.

\subsection{Maximal Grauert radius}\label{sec:grauert}

Let $S^1$ be the unit circle and $Q:  S^1 \rightarrow \p \GO \subset{\Cbb}$ be a regular real analytic parametrization of $\p\GO$. Here and afterwards we identify $\Rbb^2$ with the complex plane $\Cbb$ by the correspondence $(x_1, x_2) \mapsto x_1 + i x_2$. Then $Q$ admits an extension as an analytic function from an annulus
\beq
A_\Ge:= \{ \Gj \in {\Cbb}\; :\; e^{-\Ge}<|\Gj|<e^{\Ge}\; \} \label{analytic_annulus}
\eeq
for some $\Ge>0$ onto a tubular neighborhood of $\p\GO$ in $\Cbb$. Let
\beq
q(t):=Q(e^{it}), \quad t \in [-\pi, \pi) \times i (-\Ge, \Ge).
\eeq
Then $q$ is an analytic function from $[-\pi, \pi) \times i (-\Ge, \Ge)$ onto a tubular neighborhood of $\p\GO$.
Moreover, $q$ can be extended to $\Rbb \times i (-\Ge, \Ge)$ as an analytic $2\pi$-periodic function, namely, $q(t + 2\pi) = q(t)$.
The supremum, denoted by $\Ge_*$, of the collection of such $\Ge$ is called the {\it maximal Grauert radius} of $q$, and the set $\Rbb \times i (-\Ge_*, \Ge_*)$ the {\it maximal Grauert tube}.

In this paper we consider the numbers $\Ge$ such that $q$ satisfies an additional condition:
$$
\mbox{(G)} \quad \mbox{if $q(t)=q(s)$ for $t \in [-\pi, \pi) \times i (-\Ge, \Ge)$ and $s \in [-\pi, \pi)$, then $t=s$}.
$$
It is worth emphasizing that the condition (G) is weaker than univalence. It only requires that $q$ attains values $q(s)$ ($s \in [-\pi, \pi)$) only at $s$. The condition (G) is imposed for the integral kernel of the NP operator to be continued analytically (see \eqnref{analcon}).
We will see that this condition yields optimal convergence rate of the NP operator in examples in section \ref{sec:Examples}.

Since $Q$ is one-to-one on $\p \GO$, the extended function is univalent in $A_\Ge$ if $\Ge$ is sufficiently small.  So, the condition (G) is fulfilled if $\Ge$ is small. We denote the supremum of such $\Ge$ by $\Ge_q$. We emphasize that $\Ge_q$ may differ depending on the parametrization $q$ (see Example \ref{ex:limacon1} in section \ref{sec:Examples}). Let
\beq
\Ge_{\p\GO} := \sup_{q} \Ge_q,
\eeq
where the supremum is taken over all regular real analytic parametrization $q$ of $\p\GO$. We call $\Ge_{\p\GO}$
the {\it modified maximal Grauert radius} of $\p\GO$. The set $\Rbb \times i (-\Ge_{\p\GO}, \Ge_{\p\GO})$ is called the {\it modified maximal Grauert tube}, which we denote by $G_{\p\GO}$.

\subsection{Analytic extension of the NP operator} \label{sec:anal_ext}

Let $q$ be a regular real analytic parametrization on $[-\pi, \pi)$ of $\p\GO$. For $x, y \in \p\GO$, let $x=q(t)$ and $y=q(s)$. Then the outward unit normal vector $\nu_x$ is given by $-i q'(t)/|q'(t)|$ in the complex form. So we have
$$
\la x - y, \Gv_x \ra = \frac{1}{|q'(t)|} \Re \left[ (q(t)-q(s))\ol{(-iq'(t))} \right],
$$
and hence
$$
\frac{\la x - y, \Gv_x \ra}{| x - y |^2} =\frac{1}{2i|q'(t)|}\Big[\frac{q'(t)}{q(t)-q(s)}-\frac{\ol{q'(t)}}{\ol{q(t)} - \ol{q(s)} } \Big].
$$
So we have
$$
\Kcal_{\p \GO}^* [\Gvf] (q(t)) =
\frac{1}{4\pi i}\int_0^{2\pi} \Big[\frac{q'(s)}{q(t)-q(s)}-\frac{\ol{q'(t)}}{\ol{q(t)} - \ol{q(s)} } \Big] \, \Gvf (q(s)) ds.
$$

Define
\beq
K_q(t,s):= \frac{1}{4\pi i}\Big[\frac{q'(t)}{q(t)-q(s)}-\frac{\ol{q'(t)}}{\ol{q(t)} - \ol{q(s)} } \Big]
\eeq
and
\beq\label{Kcaldef}
\Kcal_q [f](t):= \int_{-\pi}^{\pi} K_q(t,s) f(s) ds, \quad -\pi \le t \le \pi.
\eeq
Then we have the relation
$$
\Kcal_q [\Gvf \circ q](t) =  \Kcal_{\p \GO}^* [\Gvf] (q(t)) .
$$

Let
\beq
q^*(s):= \ol{q(\ol{s})}, \quad s \in G_{\p\GO}.
\eeq
Since $q$ is analytic in $\Rbb \times i(-\Ge_q, \Ge_q)$ and satisfies the condition (G), for each fixed $t \in [-\pi, \pi)$ the kernel $K_q(t,s)$ as a function of $s$-variable has an analytic continuation to $[-\pi, \pi) \times i (-\Ge_{q}, \Ge_{q}) \setminus \{t\}$, which is given by
\beq\label{analcon}
K_q (t,s):= \frac{1}{4\pi i}\Big[\frac{q'(t)}{q(t)-q(s)}-\frac{\overline{q'(t)}}{\ol{q(t)} - q^*(s) } \Big].
\eeq
Moreover one can easily see that
\beq
\lim_{s \to t} \Big[\frac{q'(t)}{q(t)-q(s)}-\frac{\ol{q'(t)}}{\ol{q(t)} - q^*(s) } \Big] = \frac{-q'(t) \ol{q''(t)} + \ol{q'(t)} q''(t)}{2 |q'(t)|^2}. \label{removable_sing}
\eeq
Note that $q'(t) \neq 0$.
So $K_q(t,s)$ has a removable singularity at $s=t$. Thus for each fixed $t \in [-\pi, \pi)$, $K_q(t,s)$ has an analytic continuation (as a function of $s$-variable) in $[-\pi, \pi) \times i (-\Ge_{q}, \Ge_{q})$, and extends to $\Rbb \times i(-\Ge_q, \Ge_q)$ as a $2\pi$-periodic function, namely,
\beq
K_q(t, s + 2\pi) = K_q(t, s).
\eeq

Define the space $H_0$ by
\beq
H_0:= \{ \; f \; ; \; \textrm{$f(t)=(\Gvf \circ q)(t)$ ($t \in [-\pi, \pi)$) for some $\Gvf \in \Hcal^*_0$ and is $2\pi$-periodic} \; \}.
\eeq
We emphasize that $H_0$ is the collection of $2 \pi$-periodic functions in $H^{-1/2}_0[-\pi, \pi]$ equipped with the inner product inherited from $\Hcal^*_0$:
$$
\la f, g \ra_H = \la \Gvf, \Gy \ra_{\Hcal^*},
$$
where $f(t) = (\Gvf \circ q)(t)$ and $g(t) = (\Gy \circ q)(t)$ for $\Gvf, \Gy \in \Hcal^*_0$.

In the next section we look into the spectrum of $\Kcal_q$ on the space $H_0$.

\section{The main result} \label{sec:Complex_Fourier}

Let $\left\{ \Gl_n \right\}_{n = 1}^\infty$ be the eigenvalues of the NP operator $\Kcal_{\p \GO}^*$ on $H^{-1/2}_0(\p\GO)$, or equivalently, of $\Kcal_q$ on $H_0$ as defined in \eqnref{Kcaldef}. Since eigenvalues of the NP operator in two dimensions are symmetric with respect to the origin (see e.g. \cite{Blumenfeld:1914aa, MR0104934}, also \cite{HKL}), we may assume that eigenvalues are enumerated in the following way:
\beq\label{eigenvalue}
\frac{1}{2} > \left| \Gl_1 \right| = \left| \Gl_2 \right| \ge \left| \Gl_3 \right| = \left| \Gl_4 \right| \ge \cdots.
\eeq

The following theorem is the main result of this paper.

\begin{theorem}\label{thm:3}
Let $\GO$ be a bounded planar domain with the analytic boundary $\p \GO$ and $\Ge_{\p\GO}$ be the modified maximal Grauert radius of $\p \GO$.
Let $\left\{ \Gl_n \right\}_{n = 1}^\infty$ be the eigenvalues of the NP operator $\Kcal_{\p \GO}^*$ on $H^{-1/2}_0(\p\GO)$ enumerated as \eqnref{eigenvalue}. For any $\Ge < \Ge_{\p\GO}$ there is a constant $C$ such that
\beq\label{eigenest}
|\Gl_{2n-1}|=|\Gl_{2n}| \le Ce^{-n\Ge}
\eeq
for any $n$.
\end{theorem}

The rest of this section is devoted to proving Theorem \ref{thm:3}.

We first emphasize that the operator $\Kcal_q$ is symmetric on $H_0$. In fact, we have
$$
\la f, \Kcal_q[g] \ra_H = \la \Gvf, \Kcal_{\p \GO}^*[\Gy] \ra_{\Hcal^*} = \la \Kcal_{\p \GO}^*[\Gvf], \Gy \ra_{\Hcal^*} = \la \Kcal_q[f], g \ra_H,
$$
where $f(t) = (\Gvf \circ q)(t)$ and $g(t) = (\Gy \circ q)(t)$ for $\Gvf, \Gy \in \Hcal^*_0$.
Since the kernel $K_q(t, s)$ of the operator $\Kcal_q$ is $2\pi$-periodic with respect to $s$-variable, it admits the Fourier series expansion:
\beq\label{fourier}
  K_q(t, s)=\sum_{k\in {\Zbb}} a_k^q(t) e^{iks}, \quad a_k^q(t) = \frac{1}{2\pi} \int_{-\pi}^{\pi} K_q(t, s) e^{-iks} ds.
\eeq

We obtain the following lemma.

\begin{lemma}\label{lem:2}
Suppose that $\GO$ is a bounded planar domain with the analytic boundary and let $q$ be a regular real analytic parametrization on $[-\pi, \pi)$ of $\p\GO$. For any $0 < \Ge < \Ge_q$ there is a constant $C$ such that
\beq
| a_k^q(t) | \le C e^{-\Ge |k|}  \label{Fourier_coeff_estimate}
\eeq
for all integer $k$ and $t \in [-\pi, \pi)$.
\end{lemma}

\begin{proof}
If $k>0$, then we take a rectangular contour $R$ with the positive orientation in $\Rbb \times i(-\Ge_q, \Ge_q)$:
$$
R = R_1\cup R_2\cup R_3\cup R_4 := [-\pi, \pi]\cup [\pi, \pi-i\Ge] \cup [\pi-i\Ge, -\pi-i\Ge] \cup [-\pi-i \Ge, -\pi].
$$
Since $K_q(t, s)$ is analytic in $G_{\p \GO}$ and $2\pi$-periodic with respect to $s$-variable, we have
\begin{align*}
    0= \int_{R} K_q(t, s) e^{-iks} ds = & \Big\{\int_{R_1}+\int_{R_2}+\int_{R_3}+\int_{R_4}\Big\} K_q(t, s) e^{-iks} ds \\
    = & \Big\{\int_{R_1}+\int_{R_3} \Big\} K_q(t, s) e^{-iks} ds ,
\end{align*}
which implies that
\begin{equation*}
    2 \pi a_k^q(t)=\int_{R_1} K_q(t, s) e^{-iks} ds =-\int_{R_3} K_q(t, s) e^{-iks} ds =-\int_{\pi - i\Ge}^{-\pi - i\Ge} K_q(t, s) e^{-iks} ds.
\end{equation*}
  Since $K_q(t, s)$ is uniformly bounded for $t \in [-\pi, \pi)$ and $s \in R_3$, we have
  \begin{equation*}
    |a_k^q(t)| \le \left| \frac{1}{2 \pi} \int_{\pi - i\Ge}^{-\pi - i\Ge}K_q(t, s)  e^{-iks}  ds \right| \le \left| \frac{1}{2 \pi} \int_{-\pi - i\Ge}^{\pi - i\Ge} C_0 e^{-\Ge k} ds \right| = C_0 e^{-\Ge k},
  \end{equation*}
  where the constant $C_0 > 0$ is independent of $k$ and $t \in [-\pi, \pi)$.

If $k < 0$, we can prove \eqref{Fourier_coeff_estimate} by taking the rectangular contour
$$
R = [-\pi, \pi] \cup [\pi, \pi + i\Ge] \cup [\pi + i\Ge, -\pi + i\Ge] \cup [-\pi + i \Ge, -\pi].
$$
The estimate \eqref{Fourier_coeff_estimate} for $k = 0$ is obvious.
Thus the lemma follows.
\end{proof}

Let us now recall the Weyl-Courant min-max principle (see, for example, \cite{MR728688} for a proof).

\begin{theorem}[the Weyl-Courant min-max principle] \label{lem:W-C_minimax}
If $\Tcal$ is a compact symmetric operator on a Hilbert space, whose eigenvalues $\{ \Gk_n \}_{n = 1}^\infty$ are arranged as
$$
| \Gk_1 | \ge | \Gk_2 | \ge \cdots \ge | \Gk_n | \ge \cdots \to 0.
$$
If $\Scal$ is an operator of rank $\le n$, then
$$
\| \Tcal - \Scal \| \ge | \Gk_{n + 1} |.
$$
\end{theorem}

\noindent{\sl Proof of Theorem \ref{thm:3}}.
Suppose that $\Ge < \Ge_{\p\GO}$ and let $q$ be a regular real analytic parametrization of $\p\GO$ such that $\Ge < \Ge_q \le \Ge_{\p\GO}$.
Using the Fourier expansion of $K_q(t,s)$ given in \eqnref{fourier} we define
\begin{equation*}
S_n(t, s) = a_0^q(t)+\sum_{|k| \le n-1} a_{k}^q(t)e^{iks}
\end{equation*}
and
\begin{equation*}
\Scal_n[f](t) = \int_{-\pi}^{\pi} S_n(t, s) f(s) ds.
\end{equation*}
Then $\Scal_n$ is of rank at most $2(n-1)$ on $H_0$. So it follows from the Weyl-Courant min-max principle that
\beq\label{wcest}
\| \Kcal_q - \Scal_n \| \ge |\Gl_{2n}|.
\eeq

Let $f \in H_0$. It holds that
\beq
\| f \|_{-1/2}^2 \approx \sum_{k \neq 0} \frac{|\hat{f}(k)|^2}{|k|},
\eeq
where $\hat{f}(k)$ is the $k$-th Fourier coefficient of $f$, namely,
$$
\hat{f}(k) = \frac{1}{2\pi} \int_{-\pi}^{\pi} f(s) e^{-iks} ds.
$$
Note that
$$
(\Kcal_q - \Scal_n)[f](t) = 2\pi \sum_{|k| \ge n} \hat{f}(-k) a_k^q(t).
$$
So, it follows from the Cauchy-Schwarz inequality that
\begin{align*}
\left\| (\Kcal_q - \Scal_n)[f] \right\|_{-1/2} &\le C \left( \sum_{|k| \ge n} \frac{|\hat{f}(k)|^2}{|k|} \right)^{1/2}
\left( \sum_{|k| \ge n} |k| \| a_k^q \|_{-1/2}^2 \right)^{1/2} \\
&\le C \| f \|_{-1/2}
\left( \sum_{|k| \ge n} |k| \| a_k^q \|_{-1/2}^2 \right)^{1/2}
\end{align*}
for some constant $C$ which may be different at each occurrence. If $0< \Ge < \Ge' < \Ge_q$, then we obtain using \eqref{Fourier_coeff_estimate} that
$$
\sum_{|k| \ge n} |k| \| a_k^q \|_{-1/2}^2
\le  C_1 \sum_{|k| \ge n} |k| e^{-2 \Ge' |k|}
\le  C_2 \sum_{|k| \ge n}  e^{-2 \Ge |k|} \le C_3 e^{-2 \Ge n},
$$
and hence
\beq\label{est10}
\left\| (\Kcal_q - \Scal_n)[f] \right\|_{-1/2}
\le C e^{-\Ge n} \| f \|_{-1/2}.
\eeq
We then obtain \eqnref{eigenest} from \eqnref{twonorm}, \eqnref{wcest} and \eqnref{est10}. This completes the proof.
\qed

It is worth mentioning that one can also show the exponential decay of the eigenvalues for $\Kcal_q$ by using the Chebyshev expansion of $K_q(t,s)$. The Chebyshev expansion has been used in \cite{MR728688} to study eigenvalues of operators with real analytic symmetric kernels.
Using this method one can show that \eqnref{eigenest} holds for all $\Ge$ such that
\beq
\Ge < \Ge_c := \log \left( \frac{1}{\pi} \left( \Ge_{\p\GO} + \sqrt{\pi^2 + \Ge_{\p\GO}^2} \right) \right).
\eeq
This result is weaker than Theorem \ref{thm:3} since $\Ge_c < \Ge_\GO$. So, we omit the detail.

\section{Examples} \label{sec:Examples}

Theorem \ref{thm:3} shows that \eqnref{eigenest} holds for all $\Ge < \Ge_{\p\GO}$. In this section we present a few examples of domains to show that this result is optimal in the sense that $\Ge_{\p\GO}$ is the smallest number with such a property.

\begin{example}[circles] \label{ex:circle}
Suppose that $\p \GO$ is a circle. Then one can easily see that $\Ge_{\p\GO} = + \infty$. So \eqnref{eigenest} shows that for any number $\Gb>0$ there is a constant $C$ such that
$$
|\Gl_{2n}| \le C \Gb^n
$$
for all $n$. Indeed, it is known that the only eigenvalue of the NP operator on $\Hcal^*_0$ is $0$.
\end{example}

\begin{example}[ellipses] \label{ex:ellipse}
Suppose that $\p\GO$ is the ellipse given by
$$
\p \GO : \frac{x^2}{a^2}+\frac{y^2}{b^2}=1,\; a>b>0.
$$
A parametrization of $\p\GO$ is given by
$$
q(t)  = a \cos{t} + ib \sin{t} = \frac{a + b}{2} e^{i t} + \frac{a - b}{2} e^{- i t} , \quad t \in [-\pi, \pi).
$$
Note that $q$ admits analytic continuation to the whole complex plane, and so the maximal Grauert radius is $\infty$.

To compute the modified maximal Grauert radius $\Ge_q$, suppose that $q(t)= q(s)$ where $t \in [-\pi, \pi) \times i \Rbb$ and $s \in [-\pi, \pi)$.
Nontrivial solutions of this equation are given by
$$
e^{it} = \frac{a - b}{a + b} e^{-is},
$$
which implies that
$$
e^{-\Im t} = \frac{a - b}{a + b},
$$
and hence
\beq\label{eq}
    \Ge_q = \log \frac{a + b}{a - b}.
\eeq
Therefore, from Theorem~\ref{thm:3}, we have the exponential decay estimate
\beq\label{ellest}
    |\Gl_{2n-1}| = |\Gl_{2n}| \le C \Gb^n \quad \mbox{for any } \Gb > \frac{a - b}{a + b}.
\eeq

In view of \eqnref{eigell}, we see that the number $\frac{a - b}{a + b}$ in \eqnref{eq} is optimal.
It means in particular that the condition (G) is necessary for the definition of the modified maximal Grauert radius in this paper.
\end{example}

\begin{example}[lima\c{c}ons of Pascal] \label{ex:limacon1}
Let $A$ be a number such that $0 < A < \frac{1}{2}$. The lima\c{c}on of Pascal $\p\GO_A$ is defined by
\beq
\p \GO_A : \ w = z + A z^2, \  z=e^{it}, \ t \in [-\pi, \pi) .
\eeq
See Fig. 1 for the lima\c{c}on with $A=0.4$.

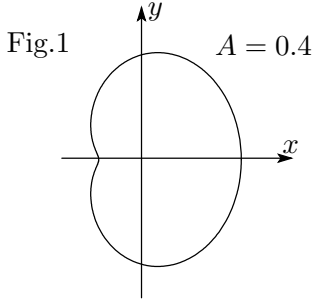
\input{pic-1.tex}

Let us first compute eigenvalues of the NP operator on $\p\GO$. For doing so, we recall that the polar equation of an ellipse with one focus at the origin is, up to similarity,
$$
r = \frac{1}{1 + e \cos{\Gt}}
$$
where $e$ is the eccentricity. Let us denote the ellipse by $\p E_e$. In complex notation $\p E_e$ is given by
$$
w= f(z):= \frac{z}{1 + e \frac{z + z^{-1}}{2}} = \frac{2}{e + 2 z^{-1} + e z^{-2}}, \quad |z|=1.
$$
Let $h$ be the bilinear transformation defined by
\beq\label{hw}
h(w):= \frac{-e w+2}{2w}.
\eeq
Then we have
\beq\label{hfz}
h(f(z)) = z^{-1} + \frac{e}{2} z^{-2}.
\eeq
This is the lima\c{c}on with $A = \frac{e}{2}$. In short we have
\beq\label{hEGO}
h(\p E_{2A})= \p\GO_A.
\eeq
According to \cite[p.1195]{MR0104934}, eigenvalues of the NP operator are invariant under bilinear transformations, and hence NP operators on $\p E_{2A}$ and $\p\GO_A$ have identical eigenvalues. In view of \eqnref{eigell}, we see that eigenvalues of the NP operator on $\p\GO_A$ are
\beq\label{eiglima}
\pm \frac{1}{2}\left( \frac{1 - \sqrt{1 - 4 A^2}}{1 + \sqrt{1 - 4 A^2}} \right)^n.
\eeq

A straight-forward parametrization of the lima\c{c}on $\p\GO_A$ is given by
\beq
q(t):= e^{it} + A e^{2it}, \quad t \in [-\pi, \pi).
\eeq
So, $q$ can be extended analytically to the whole complex plane. To find $\Ge_q$ we suppose $q(t)=q(s)$ for some $t \in [-\pi, \pi) \times i \Rbb$ and $s \in [-\pi, \pi)$. Then non-trivial solutions are $e^{it} = -e^{is} - 1/A$, and hence $e^{-\Im t}= |e^{is} +1/A|$. Therefore, we have
$$
\Ge_q = \inf_{s} \log \left| e^{is} + \frac{1}{A} \right| = \log \left( \frac{1}{A} - 1 \right).
$$
So we infer from Theorem~\ref{thm:3} that
  \begin{equation*}
    |\Gl_{2n-1}| = |\Gl_{2n}| \le C \Gb^n \quad \mbox{for any } \Gb > \frac{A}{1 - A}.
  \end{equation*}
One can see from \eqnref{eiglima} that this estimate is not optimal since
$$
\frac{1 - \sqrt{1 - 4 A^2}}{1 + \sqrt{1 - 4 A^2}} < \frac{A}{1 - A} .
$$

However, we may use another parametrization of $\p\GO_A$ to obtain an optimal estimate. In fact, let $e=2A$, and
$$
a:= \frac{1}{1-e^2}, \quad b:= a \sqrt{1-e^2}
$$
so that
$$
g(z) = \frac{a + b}{2} z + \frac{a - b}{2} z^{-1} + a e, \quad |z|=1
$$
is a complex parametrization of $\p E_e$. Using the bilinear transformation $h$ in \eqnref{hw}, define
\beq
q_1(t):= h(g(e^{it})) .
\eeq
Then, \eqnref{hEGO} shows that $q_1(t)$, $t \in [-\pi, \pi)$, is a parametrization of $\p\GO_A$.

If $q_1(t)=q_1(s)$, then $g(e^{it})=g(e^{is})$. So, as shown in Example~\ref{ex:ellipse}, we have
$$
\Ge_{q_1} = \log \frac{a + b}{a - b} =\log \frac{1 + \sqrt{1 -4 A^2}}{1 - \sqrt{1 - 4 A^2}},
$$
which yields an optimal estimate.
\end{example}

It is worth mentioning that all three examples above show that \eqnref{eigenest} holds even for $\Ge=\Ge_{\p\GO}$, namely, there is a constant $C$ such that
\beq
|\Gl_{2n}| \le C e^{-n\Ge_{\p\GO}}
\eeq
for all $n$. It is interesting to prove this.

Let us present one more example of a curve on which the NP eigenvalues are not known.

\begin{example}[the transcendental curves]
We consider the transcendental curve
\begin{equation*}
\p \GO : w = e^{A z}, \ |z|=1, \  0 < |A| < \pi .
    \quad \mbox{(See Fig.2).}
\end{equation*}
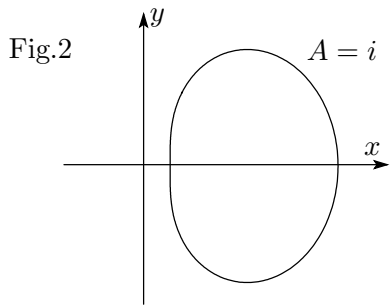
\input{pic-3.tex}

An obvious parametrization of $\p\GO$ is given by
$$
q(t):= \exp(Ae^{it}), \quad t \in [-\pi, \pi).
$$
If $q(t)=q(s)$ for some $t \in [-\pi, \pi) \times i \Rbb$ and $s \in [-\pi, \pi)$, then non-trivial solutions are given by
$$
A e^{it} = A e^{is} + i 2\pi n, \quad n \in \Zbb \ (n \neq 0).
$$
So we have
$$
\Ge_q= \inf_{n,s} \log \left| e^{is} + \frac{i2\pi n}{A} \right|= \log \left( \frac{2\pi}{|A|}-1 \right).
$$
Thus, we have
\beq
    |\Gl_{2n-1}| = |\Gl_{2n}| \le C \Gb^n \quad\mbox{for any } \Gb > \frac{|A|}{2\pi-|A|}.
\eeq
\end{example}

\end{document}

%% file: pic-1.tex
{\unitlength 0.1in%
\begin{picture}(15.2900,15.8400)(21.0000,-30.9000)%
%
\special{pn 8}%
\special{pa 2428 2359}%
\special{pa 3629 2359}%
\special{fp}%
\special{sh 1}%
\special{pa 3629 2359}%
\special{pa 3562 2339}%
\special{pa 3576 2359}%
\special{pa 3562 2379}%
\special{pa 3629 2359}%
\special{fp}%
%
\special{pn 8}%
\special{pa 2845 3090}%
\special{pa 2845 1570}%
\special{fp}%
\special{sh 1}%
\special{pa 2845 1570}%
\special{pa 2825 1637}%
\special{pa 2845 1623}%
\special{pa 2865 1637}%
\special{pa 2845 1570}%
\special{fp}%
\special{pn 8}%
\special{pa 3367 2366}%
\special{pa 3367 2335}%
\special{pa 3366 2324}%
\special{pa 3364 2304}%
\special{pa 3363 2293}%
\special{pa 3361 2273}%
\special{pa 3360 2262}%
\special{pa 3348 2202}%
\special{pa 3345 2192}%
\special{pa 3343 2182}%
\special{pa 3340 2172}%
\special{pa 3337 2163}%
\special{pa 3331 2143}%
\special{pa 3328 2134}%
\special{pa 3324 2125}%
\special{pa 3320 2115}%
\special{pa 3317 2106}%
\special{pa 3301 2070}%
\special{pa 3296 2061}%
\special{pa 3292 2053}%
\special{pa 3287 2044}%
\special{pa 3283 2036}%
\special{pa 3263 2004}%
\special{pa 3257 1996}%
\special{pa 3252 1988}%
\special{pa 3246 1981}%
\special{pa 3241 1974}%
\special{pa 3235 1966}%
\special{pa 3229 1959}%
\special{pa 3224 1952}%
\special{pa 3218 1945}%
\special{pa 3212 1939}%
\special{pa 3206 1932}%
\special{pa 3199 1926}%
\special{pa 3187 1914}%
\special{pa 3180 1908}%
\special{pa 3174 1902}%
\special{pa 3167 1897}%
\special{pa 3161 1891}%
\special{pa 3133 1871}%
\special{pa 3126 1867}%
\special{pa 3120 1862}%
\special{pa 3112 1858}%
\special{pa 3084 1842}%
\special{pa 3070 1836}%
\special{pa 3062 1833}%
\special{pa 3048 1827}%
\special{pa 3040 1824}%
\special{pa 3026 1820}%
\special{pa 3018 1818}%
\special{pa 3004 1814}%
\special{pa 2996 1813}%
\special{pa 2989 1811}%
\special{pa 2982 1810}%
\special{pa 2974 1809}%
\special{pa 2960 1807}%
\special{pa 2952 1807}%
\special{pa 2945 1806}%
\special{pa 2916 1806}%
\special{pa 2909 1807}%
\special{pa 2902 1807}%
\special{pa 2867 1812}%
\special{pa 2860 1814}%
\special{pa 2853 1815}%
\special{pa 2847 1817}%
\special{pa 2833 1821}%
\special{pa 2827 1823}%
\special{pa 2820 1826}%
\special{pa 2814 1828}%
\special{pa 2807 1831}%
\special{pa 2801 1833}%
\special{pa 2771 1848}%
\special{pa 2765 1852}%
\special{pa 2759 1855}%
\special{pa 2753 1859}%
\special{pa 2748 1863}%
\special{pa 2742 1866}%
\special{pa 2737 1870}%
\special{pa 2731 1874}%
\special{pa 2726 1878}%
\special{pa 2721 1883}%
\special{pa 2711 1891}%
\special{pa 2706 1896}%
\special{pa 2701 1900}%
\special{pa 2696 1905}%
\special{pa 2691 1909}%
\special{pa 2687 1914}%
\special{pa 2682 1919}%
\special{pa 2650 1959}%
\special{pa 2647 1965}%
\special{pa 2643 1970}%
\special{pa 2640 1975}%
\special{pa 2637 1981}%
\special{pa 2633 1986}%
\special{pa 2630 1991}%
\special{pa 2627 1997}%
\special{pa 2624 2002}%
\special{pa 2622 2008}%
\special{pa 2619 2013}%
\special{pa 2616 2019}%
\special{pa 2614 2024}%
\special{pa 2612 2030}%
\special{pa 2609 2036}%
\special{pa 2607 2041}%
\special{pa 2605 2047}%
\special{pa 2603 2052}%
\special{pa 2601 2058}%
\special{pa 2599 2063}%
\special{pa 2597 2069}%
\special{pa 2596 2074}%
\special{pa 2594 2080}%
\special{pa 2593 2085}%
\special{pa 2591 2091}%
\special{pa 2589 2101}%
\special{pa 2587 2107}%
\special{pa 2585 2117}%
\special{pa 2585 2123}%
\special{pa 2582 2138}%
\special{pa 2582 2143}%
\special{pa 2581 2148}%
\special{pa 2581 2153}%
\special{pa 2580 2158}%
\special{pa 2580 2168}%
\special{pa 2579 2173}%
\special{pa 2579 2196}%
\special{pa 2580 2201}%
\special{pa 2580 2214}%
\special{pa 2581 2218}%
\special{pa 2581 2226}%
\special{pa 2582 2230}%
\special{pa 2582 2234}%
\special{pa 2584 2242}%
\special{pa 2584 2246}%
\special{pa 2585 2249}%
\special{pa 2585 2253}%
\special{pa 2586 2256}%
\special{pa 2587 2260}%
\special{pa 2588 2263}%
\special{pa 2588 2267}%
\special{pa 2591 2276}%
\special{pa 2592 2280}%
\special{pa 2594 2286}%
\special{pa 2594 2288}%
\special{pa 2597 2297}%
\special{pa 2598 2299}%
\special{pa 2599 2302}%
\special{pa 2600 2304}%
\special{pa 2601 2307}%
\special{pa 2602 2309}%
\special{pa 2603 2312}%
\special{pa 2603 2314}%
\special{pa 2609 2326}%
\special{pa 2609 2328}%
\special{pa 2611 2332}%
\special{pa 2612 2333}%
\special{pa 2612 2335}%
\special{pa 2613 2337}%
\special{pa 2614 2338}%
\special{pa 2614 2340}%
\special{pa 2615 2342}%
\special{pa 2616 2343}%
\special{pa 2616 2345}%
\special{pa 2617 2346}%
\special{pa 2617 2347}%
\special{pa 2618 2349}%
\special{pa 2618 2350}%
\special{pa 2619 2351}%
\special{pa 2620 2358}%
\special{pa 2621 2359}%
\special{pa 2621 2373}%
\special{pa 2620 2374}%
\special{pa 2620 2377}%
\special{pa 2619 2378}%
\special{pa 2619 2380}%
\special{pa 2618 2382}%
\special{pa 2618 2383}%
\special{pa 2617 2384}%
\special{pa 2616 2389}%
\special{pa 2615 2390}%
\special{pa 2615 2392}%
\special{pa 2614 2393}%
\special{pa 2613 2395}%
\special{pa 2613 2396}%
\special{pa 2610 2402}%
\special{pa 2610 2403}%
\special{pa 2605 2413}%
\special{pa 2605 2415}%
\special{pa 2604 2418}%
\special{pa 2602 2422}%
\special{pa 2601 2425}%
\special{pa 2599 2429}%
\special{pa 2597 2435}%
\special{pa 2596 2437}%
\special{pa 2596 2440}%
\special{pa 2589 2461}%
\special{pa 2589 2464}%
\special{pa 2588 2468}%
\special{pa 2587 2471}%
\special{pa 2586 2475}%
\special{pa 2586 2478}%
\special{pa 2585 2482}%
\special{pa 2584 2485}%
\special{pa 2584 2489}%
\special{pa 2583 2493}%
\special{pa 2583 2497}%
\special{pa 2582 2501}%
\special{pa 2582 2505}%
\special{pa 2581 2509}%
\special{pa 2581 2513}%
\special{pa 2580 2517}%
\special{pa 2580 2530}%
\special{pa 2579 2535}%
\special{pa 2579 2558}%
\special{pa 2580 2563}%
\special{pa 2580 2573}%
\special{pa 2581 2577}%
\special{pa 2581 2582}%
\special{pa 2582 2588}%
\special{pa 2582 2593}%
\special{pa 2584 2603}%
\special{pa 2584 2608}%
\special{pa 2585 2613}%
\special{pa 2586 2619}%
\special{pa 2588 2629}%
\special{pa 2590 2635}%
\special{pa 2591 2640}%
\special{pa 2592 2646}%
\special{pa 2594 2651}%
\special{pa 2595 2656}%
\special{pa 2597 2662}%
\special{pa 2599 2667}%
\special{pa 2600 2673}%
\special{pa 2602 2678}%
\special{pa 2606 2690}%
\special{pa 2609 2695}%
\special{pa 2611 2701}%
\special{pa 2613 2706}%
\special{pa 2616 2712}%
\special{pa 2618 2717}%
\special{pa 2621 2723}%
\special{pa 2624 2728}%
\special{pa 2627 2734}%
\special{pa 2630 2739}%
\special{pa 2633 2745}%
\special{pa 2639 2755}%
\special{pa 2642 2761}%
\special{pa 2646 2766}%
\special{pa 2649 2771}%
\special{pa 2653 2777}%
\special{pa 2681 2812}%
\special{pa 2709 2840}%
\special{pa 2719 2848}%
\special{pa 2725 2853}%
\special{pa 2735 2861}%
\special{pa 2741 2865}%
\special{pa 2746 2868}%
\special{pa 2752 2872}%
\special{pa 2757 2876}%
\special{pa 2763 2879}%
\special{pa 2769 2883}%
\special{pa 2799 2898}%
\special{pa 2806 2901}%
\special{pa 2812 2903}%
\special{pa 2819 2906}%
\special{pa 2825 2908}%
\special{pa 2832 2910}%
\special{pa 2838 2912}%
\special{pa 2852 2916}%
\special{pa 2858 2918}%
\special{pa 2865 2919}%
\special{pa 2872 2921}%
\special{pa 2900 2925}%
\special{pa 2907 2925}%
\special{pa 2914 2926}%
\special{pa 2943 2926}%
\special{pa 2950 2925}%
\special{pa 2958 2925}%
\special{pa 2972 2923}%
\special{pa 2980 2922}%
\special{pa 2994 2920}%
\special{pa 3002 2918}%
\special{pa 3009 2917}%
\special{pa 3017 2915}%
\special{pa 3031 2911}%
\special{pa 3039 2908}%
\special{pa 3046 2906}%
\special{pa 3060 2900}%
\special{pa 3068 2897}%
\special{pa 3075 2894}%
\special{pa 3082 2890}%
\special{pa 3089 2887}%
\special{pa 3097 2883}%
\special{pa 3125 2867}%
\special{pa 3139 2857}%
\special{pa 3145 2852}%
\special{pa 3166 2837}%
\special{pa 3172 2831}%
\special{pa 3179 2826}%
\special{pa 3191 2814}%
\special{pa 3198 2808}%
\special{pa 3204 2801}%
\special{pa 3210 2795}%
\special{pa 3222 2781}%
\special{pa 3228 2775}%
\special{pa 3234 2768}%
\special{pa 3239 2760}%
\special{pa 3251 2746}%
\special{pa 3286 2690}%
\special{pa 3291 2681}%
\special{pa 3295 2673}%
\special{pa 3300 2664}%
\special{pa 3320 2619}%
\special{pa 3323 2610}%
\special{pa 3327 2600}%
\special{pa 3330 2591}%
\special{pa 3333 2581}%
\special{pa 3336 2572}%
\special{pa 3342 2552}%
\special{pa 3345 2543}%
\special{pa 3347 2533}%
\special{pa 3350 2523}%
\special{pa 3356 2493}%
\special{pa 3358 2482}%
\special{pa 3359 2472}%
\special{pa 3361 2462}%
\special{pa 3362 2452}%
\special{pa 3363 2441}%
\special{pa 3365 2421}%
\special{pa 3366 2410}%
\special{pa 3366 2400}%
\special{pa 3367 2389}%
\special{pa 3367 2369}%
\special{fp}%
\put(32.3000,-18.2000){\makebox(0,0)[lb]{$A=0.4$}}%
\put(35.8400,-23.3300){\makebox(0,0)[lb]{$x$}}%
\put(28.7800,-16.3600){\makebox(0,0)[lb]{$y$}}%
\put(21.4000,-18.3000){\makebox(0,0)[lb]{Fig.1}}%
\end{picture}}%

%% file: pic-3.tex
{\unitlength 0.1in%
\begin{picture}(25.1000,15.8400)(21.0000,-30.9000)%
%
\special{pn 8}%
\special{pa 2428 2359}%
\special{pa 4120 2359}%
\special{fp}%
\special{sh 1}%
\special{pa 4120 2359}%
\special{pa 4053 2339}%
\special{pa 4067 2359}%
\special{pa 4053 2379}%
\special{pa 4120 2359}%
\special{fp}%
%
\special{pn 8}%
\special{pa 2845 3090}%
\special{pa 2845 1570}%
\special{fp}%
\special{sh 1}%
\special{pa 2845 1570}%
\special{pa 2825 1637}%
\special{pa 2845 1623}%
\special{pa 2865 1637}%
\special{pa 2845 1570}%
\special{fp}%
\special{pn 8}%
\special{pa 3048 1977}%
\special{pa 3046 1982}%
\special{pa 3043 1987}%
\special{pa 3039 1997}%
\special{pa 3036 2001}%
\special{pa 3030 2016}%
\special{pa 3028 2020}%
\special{pa 3026 2025}%
\special{pa 3025 2030}%
\special{pa 3023 2034}%
\special{pa 3021 2039}%
\special{pa 3020 2043}%
\special{pa 3018 2048}%
\special{pa 3017 2052}%
\special{pa 3015 2056}%
\special{pa 3014 2061}%
\special{pa 3012 2065}%
\special{pa 3011 2069}%
\special{pa 3010 2074}%
\special{pa 3008 2082}%
\special{pa 3006 2086}%
\special{pa 3002 2102}%
\special{pa 3002 2106}%
\special{pa 2998 2122}%
\special{pa 2998 2125}%
\special{pa 2996 2133}%
\special{pa 2996 2136}%
\special{pa 2994 2144}%
\special{pa 2994 2147}%
\special{pa 2993 2151}%
\special{pa 2993 2154}%
\special{pa 2992 2158}%
\special{pa 2992 2161}%
\special{pa 2991 2164}%
\special{pa 2991 2168}%
\special{pa 2990 2171}%
\special{pa 2990 2178}%
\special{pa 2989 2181}%
\special{pa 2989 2187}%
\special{pa 2988 2190}%
\special{pa 2988 2199}%
\special{pa 2987 2202}%
\special{pa 2987 2211}%
\special{pa 2986 2214}%
\special{pa 2986 2228}%
\special{pa 2985 2231}%
\special{pa 2985 2252}%
\special{pa 2984 2255}%
\special{pa 2984 2476}%
\special{pa 2985 2479}%
\special{pa 2985 2500}%
\special{pa 2986 2503}%
\special{pa 2986 2517}%
\special{pa 2987 2520}%
\special{pa 2987 2532}%
\special{pa 2988 2535}%
\special{pa 2988 2541}%
\special{pa 2989 2544}%
\special{pa 2989 2550}%
\special{pa 2990 2553}%
\special{pa 2990 2560}%
\special{pa 2991 2563}%
\special{pa 2991 2567}%
\special{pa 2992 2570}%
\special{pa 2992 2573}%
\special{pa 2993 2577}%
\special{pa 2993 2580}%
\special{pa 2994 2584}%
\special{pa 2994 2587}%
\special{pa 2995 2591}%
\special{pa 2995 2594}%
\special{pa 2997 2602}%
\special{pa 2997 2605}%
\special{pa 3000 2617}%
\special{pa 3000 2621}%
\special{pa 3002 2629}%
\special{pa 3003 2632}%
\special{pa 3004 2636}%
\special{pa 3005 2641}%
\special{pa 3008 2653}%
\special{pa 3010 2657}%
\special{pa 3011 2661}%
\special{pa 3012 2666}%
\special{pa 3013 2670}%
\special{pa 3015 2674}%
\special{pa 3016 2679}%
\special{pa 3018 2683}%
\special{pa 3019 2687}%
\special{pa 3021 2692}%
\special{pa 3022 2696}%
\special{pa 3026 2706}%
\special{pa 3028 2710}%
\special{pa 3032 2720}%
\special{pa 3034 2724}%
\special{pa 3040 2739}%
\special{pa 3042 2743}%
\special{pa 3045 2748}%
\special{pa 3047 2753}%
\special{pa 3050 2758}%
\special{pa 3052 2763}%
\special{pa 3073 2798}%
\special{pa 3077 2803}%
\special{pa 3080 2808}%
\special{pa 3088 2818}%
\special{pa 3091 2823}%
\special{pa 3103 2838}%
\special{pa 3108 2843}%
\special{pa 3116 2853}%
\special{pa 3135 2872}%
\special{pa 3141 2877}%
\special{pa 3146 2882}%
\special{pa 3151 2886}%
\special{pa 3157 2891}%
\special{pa 3162 2896}%
\special{pa 3174 2904}%
\special{pa 3180 2909}%
\special{pa 3198 2921}%
\special{pa 3205 2925}%
\special{pa 3211 2929}%
\special{pa 3218 2933}%
\special{pa 3225 2936}%
\special{pa 3232 2940}%
\special{pa 3239 2943}%
\special{pa 3246 2947}%
\special{pa 3254 2950}%
\special{pa 3261 2953}%
\special{pa 3269 2956}%
\special{pa 3277 2958}%
\special{pa 3285 2961}%
\special{pa 3317 2969}%
\special{pa 3326 2971}%
\special{pa 3335 2972}%
\special{pa 3343 2974}%
\special{pa 3352 2975}%
\special{pa 3361 2975}%
\special{pa 3370 2976}%
\special{pa 3407 2976}%
\special{pa 3416 2975}%
\special{pa 3436 2973}%
\special{pa 3445 2972}%
\special{pa 3465 2968}%
\special{pa 3474 2966}%
\special{pa 3484 2963}%
\special{pa 3494 2961}%
\special{pa 3504 2957}%
\special{pa 3514 2954}%
\special{pa 3544 2942}%
\special{pa 3564 2932}%
\special{pa 3573 2927}%
\special{pa 3613 2903}%
\special{pa 3622 2896}%
\special{pa 3632 2888}%
\special{pa 3641 2881}%
\special{pa 3651 2873}%
\special{pa 3660 2865}%
\special{pa 3669 2856}%
\special{pa 3678 2848}%
\special{pa 3687 2838}%
\special{pa 3696 2829}%
\special{pa 3705 2819}%
\special{pa 3721 2799}%
\special{pa 3730 2788}%
\special{pa 3738 2777}%
\special{pa 3745 2766}%
\special{pa 3761 2742}%
\special{pa 3775 2718}%
\special{pa 3782 2705}%
\special{pa 3795 2679}%
\special{pa 3801 2666}%
\special{pa 3806 2653}%
\special{pa 3812 2639}%
\special{pa 3822 2611}%
\special{pa 3827 2596}%
\special{pa 3832 2582}%
\special{pa 3840 2552}%
\special{pa 3843 2537}%
\special{pa 3847 2522}%
\special{pa 3850 2507}%
\special{pa 3852 2491}%
\special{pa 3855 2476}%
\special{pa 3857 2460}%
\special{pa 3859 2445}%
\special{pa 3862 2397}%
\special{pa 3862 2382}%
\special{pa 3863 2366}%
\special{pa 3862 2350}%
\special{pa 3862 2334}%
\special{pa 3861 2318}%
\special{pa 3860 2303}%
\special{pa 3859 2287}%
\special{pa 3857 2271}%
\special{pa 3855 2256}%
\special{pa 3852 2240}%
\special{pa 3850 2225}%
\special{pa 3847 2210}%
\special{pa 3843 2195}%
\special{pa 3840 2180}%
\special{pa 3832 2150}%
\special{pa 3827 2135}%
\special{pa 3812 2093}%
\special{pa 3806 2079}%
\special{pa 3801 2066}%
\special{pa 3794 2052}%
\special{pa 3782 2026}%
\special{pa 3775 2014}%
\special{pa 3768 2001}%
\special{pa 3760 1989}%
\special{pa 3753 1977}%
\special{pa 3721 1933}%
\special{pa 3713 1923}%
\special{pa 3704 1913}%
\special{pa 3696 1903}%
\special{pa 3687 1893}%
\special{pa 3669 1875}%
\special{pa 3660 1867}%
\special{pa 3650 1859}%
\special{pa 3632 1843}%
\special{pa 3612 1829}%
\special{pa 3603 1823}%
\special{pa 3593 1816}%
\special{pa 3583 1811}%
\special{pa 3573 1805}%
\special{pa 3553 1795}%
\special{pa 3544 1790}%
\special{pa 3514 1778}%
\special{pa 3504 1775}%
\special{pa 3494 1771}%
\special{pa 3484 1769}%
\special{pa 3474 1766}%
\special{pa 3464 1764}%
\special{pa 3455 1762}%
\special{pa 3445 1760}%
\special{pa 3435 1759}%
\special{pa 3426 1758}%
\special{pa 3416 1757}%
\special{pa 3407 1756}%
\special{pa 3370 1756}%
\special{pa 3361 1757}%
\special{pa 3352 1757}%
\special{pa 3343 1759}%
\special{pa 3334 1760}%
\special{pa 3326 1761}%
\special{pa 3317 1763}%
\special{pa 3301 1767}%
\special{pa 3292 1769}%
\special{pa 3284 1771}%
\special{pa 3277 1774}%
\special{pa 3269 1776}%
\special{pa 3261 1779}%
\special{pa 3254 1782}%
\special{pa 3246 1785}%
\special{pa 3239 1789}%
\special{pa 3232 1792}%
\special{pa 3225 1796}%
\special{pa 3218 1799}%
\special{pa 3211 1803}%
\special{pa 3205 1807}%
\special{pa 3198 1811}%
\special{pa 3186 1819}%
\special{pa 3179 1823}%
\special{pa 3173 1828}%
\special{pa 3168 1832}%
\special{pa 3162 1837}%
\special{pa 3156 1841}%
\special{pa 3151 1846}%
\special{pa 3146 1850}%
\special{pa 3140 1855}%
\special{pa 3116 1879}%
\special{pa 3108 1889}%
\special{pa 3103 1894}%
\special{pa 3091 1909}%
\special{pa 3088 1914}%
\special{pa 3080 1924}%
\special{pa 3077 1929}%
\special{pa 3073 1934}%
\special{pa 3052 1969}%
\special{pa 3050 1974}%
\special{pa 3047 1979}%
\special{pa 3045 1984}%
\special{fp}%
\put(37.0000,-18.1000){\makebox(0,0)[lb]{$A=i$}}%
\put(40.0000,-23.1000){\makebox(0,0)[lb]{$x$}}%
\put(28.7800,-16.3600){\makebox(0,0)[lb]{$y$}}%
\put(21.4000,-18.3000){\makebox(0,0)[lb]{Fig.2}}%
\end{picture}}%